\UseRawInputEncoding
\documentclass[11pt]{article}
\usepackage{amsmath}
\usepackage{amssymb, amscd, amsthm,amsfonts}
\usepackage[all]{xy}
\usepackage[dvips]{graphicx}
\usepackage{verbatim}
\usepackage[perpage,symbol*]{footmisc}
\newtheorem{theorem}{Theorem}
\newtheorem{lemma}{Lemma}

\newtheorem{corollary}{Corollary}
\newtheorem{conjecture}{Conjecture}

\newtheorem{exam}{Example}

\begin{document}
	
	\pagestyle{myheadings}

	\title{\bf Girth of the algebraic bipartite graph $D(k,q)$ \footnote{This work was supported by the National Natural Science Foundation of China (No. 61977056).}}
	\author{Ming Xu, Xiaoyan Cheng and Yuansheng Tang\footnote{Corresponding author.}\\
		{\it\small School of Mathematical Sciences, Yangzhou University, Jiangsu, China\footnote{Email addresses: mx120170247@yzu.edu.cn(M. Xu), xycheng@yzu.edu.cn(X. Cheng), ystang@yzu.edu.cn(Y. Tang)}}
	}
	\date{}
	\maketitle
	
	{\noindent\small{\bf Abstract:}
		For integer $k\geq2$ and prime power $q$, the algebraic bipartite graph $D(k,q)$ proposed by Lazebnik and Ustimenko (1995) is meaningful not only in extremal graph theory but also in coding theory and cryptography. This graph is $q$-regular, edge-transitive and of girth at least $k+4$.
		Its exact girth $g=g(D(k,q))$ was conjectured in 1995 to be $k+5$ for odd $k$ and $q\geq4$.
		This conjecture was shown to be valid in 2016 when $\frac{k+5}{2}|_p(q-1)$, where $p$ is the characteristic of $\mathbb{F}_q$ and $m|_pn$ means that $m$ divides $p^r n$ for some nonnegative integer $r$. In this paper, for $t\geq 1$ we prove that
(a) $g(D(4t+2,q))=g(D(4t+1,q))$;
(b) $g(D(4t+3,q))=4t+8$ if $g(D(2t,q))=2t+4$;
(c) $g(D(8t,q))=8t+4$ if $g(D(4t-2,q))=4t+2$;
(d) $g(D(2^{s+2}(2t-1)-5,q))=2^{s+2}(2t-1)$ if $p\geq 3$, $(2t-1)|_p(q-1)$ and $2^s\|(q-1)$.
	A simple upper bound for the girth of $D(k,q)$ is proposed in the end of this paper. }
	
	\vspace{1ex}
	
\section{Introduction}
The graphs considered in this paper are undirected, without loops and multiple edges.
   A graph $G$ is said to be {\it edge-transitive} provided, for any two edges $e_1$, $e_2$ of $G$, that there exists an automorphism $\phi$ of $G$ such that $\phi$ maps the ends of $e_1$ into those of $e_2$.
   A {\it backtrackless} (or {\it non-recurrent}) walk of length $n$ means a sequence $v_1,v_2,\ldots,v_n$ of vertices of $G$ such that any two consecutive vertices are adjacent in $G$ and $v_j\neq v_{j+2}$ for $j=1,2,\ldots,n-2$.
   A backtrackless walk $v_1,v_2,\ldots,v_n$ is called a {\it backtrackless circuit} of length $n$ further if $n$ is greater than 2 and $v_3,v_4,\ldots,v_n,v_1,v_2$ is still a backtrackless walk.
   For any graph $G$ which is not a tree, its girth, denoted by $g(G)$, is equal to the length of the shortest
   backtrackless circuits in $G$.

   Graphs with large girth and a high degree of symmetry have been applied to variant problems in extremal graph theory, finite geometry, coding theory, cryptography, communication networks and quantum computations (c.f. \cite{Wenger91}--\cite{Dehghan20}).
   In this paper, we concetrate on the algebraic bipartite graph $D(k,q)$, proposed by Lazebnik and Ustimenko in \cite{Lazebnik95}, which is edge-transitive and of girth at least $k+4$, where $k\geq2$ and $q=p^m$ is a power of prime $p$. The graph $D(k,q)$ has been investigated quite well in literature (c.f. \cite{Lazebnik95}--\cite{XWY16}).
   For the exact girth of $D(k,q)$, the following conjecture was proposed in \cite{Fredi95}:
   \begin{conjecture}\label{conmain}
     $D(k,q)$ has girth $k+5$ for all odd $k$ and all $q\geq4$.
   \end{conjecture}
   \noindent
   This conjecture was shown to be valid in \cite{Fredi95} for the case that $(k+5)/2$ is a factor of $q-1$, in \cite{XWY14} for the case that $(k+5)/2$ is a power of $p$, and in \cite{XWY16} for the case that $(k+5)/2$ is a factor of $q-1$ multiplied by a power of $p$, respectively.
   For a few small $k$'s, the girth cycles (namely, the shortest backtrackless circuits) of $D(k,q)$ are determined completely in \cite{XCT22}.

   In this paper, we will investigate the girth of $D(k,q)$ further by means of a compact expression of some backtrackless walks of the bipartite graph
$\Lambda_{k,q}$, which is defined as follows (c.f. \cite{XWY14}, \cite{XWY16} and \cite{XCT22}). The left part of vertices of $\Lambda_{k,q}$, denoted by $L_k$, is
   the set of $(k+1)$-dimensional vectors $[l]=(l_0,l_1,l_2,\ldots,l_k)$ over $\mathbb{F}_q$ with $l_1=l_2$.
   The right part of vertices of $\Lambda_{k,q}$, denoted by $R_k$, is the set of $(k+1)$-dimensional vectors $\langle r \rangle=(r_0,r_1,r_2,\ldots,r_k)$ over $\mathbb{F}_q$ with $r_1=0$. Two vertices $(l_0,l_1,\ldots,l_k)\in L_k$ and $(r_0,r_1,\ldots,r_k)\in R_k$ are adjacent in $\Lambda_{k,q}$ if and only if, for $2\leq i \leq k$,
   \begin{align}\label{adj_con}
      l_i+r_i=\begin{cases}
         r_0l_{i-2} \text{ ~~if } i\equiv2,3 \mod 4,\\
         l_0r_{i-2} \text{ ~~if } i\equiv0,1 \mod 4.
      \end{cases}
   \end{align}
Since $\Lambda_{k,q}$ is isomorphic to $D(k,q)$ \cite{XWY14}, $\Lambda_{k,q}$ is also edge-transitive and of girth at least $k+4$.
All the consequent arguments will be made on the graph $\Lambda_{k,q}$ instead of the original graph $D(k,q)$.

   This paper is arranged as follows.
   In Section~2 we introduce a class of homogeneous polynomials in several indeterminates and a compact expression for vertices over some backtrackless walks in $\Lambda_{k,q}$. An identity on such polynomials is shown in Section~3. By using of this identity, in Section~3 we show, for any $t\geq 1$, that each backtrackless circuit in $\Lambda_{4t+1,q}$ implies
   a backtrackless circuit in $\Lambda_{4t+2,q}$ of the same length, and $g(\Lambda_{4t+2,q})=g(\Lambda_{4t+1,q})$ is then deduced.
   In Section~4, we construct some backtrackless circuits in $\Lambda_{4t+3,q}$ by using those in $\Lambda_{2t,q}$, and show for $n\geq 3$ that $g(\Lambda_{4t+3,q})\leq 4n$ if $g(\Lambda_{2t,q})\leq 2n$. A few results on the exact girth of $\Lambda_{k,q}$ are also given in this section.
   In Section~5, we deduce an upper bound for the girth of $\Lambda_{k,q}$ by combining a known result from \cite{XWY16} and the new results shown in Section~4.
   Some concluding remarks are given in Section~6.
   	
\section{Backtrackless Walks in $\Lambda_{k,q}$}	
	At first, we introduce a class of homogeneous polynomials in several indeterminates which were defined in \cite{XWY14}.
For indeterminates $\omega_1,\ldots,\omega_n$ whose values are usually limited to the set $\mathbb{F}_q^*$, let
	$$ \rho_0(\omega_1,\ldots,\omega_n)=\omega_1\cdots\omega_n$$
	and, for $1 \leq s\leq \lfloor\frac{n}{2}\rfloor$, let $\rho_s(\omega_1,\ldots,\omega_n)$ denote the homogeneous polynomial of order $n-2s$ defined by
	$$ \rho_s(\omega_1,\ldots,\omega_n)=\sum_{1\leq i_1<\cdots<i_s\leq n-s}\frac{\prod_{j=1}^n\omega_j}{\prod_{j=1}^s\omega_{i_j+j-1}\omega_{i_j+j}},$$
	where each term in the summation is a product of the remaining elements in the sequence $\omega_1,\ldots,\omega_n$ after deleting from it $s$ disjoint pairs $\{\omega_i,\omega_{i+1}\}$ of consecutive elements.
	If $n<2s$ or $s<0$, $\rho_s(\omega_1,\ldots,\omega_n)$ is defined as 0. For the null sequence $\eta$, we define $\rho_s(\eta)$ as
	\begin{align*}
		\rho_s(\eta)=\begin{cases}
			1 \text{ ~~if } s=0,\\
			0 \text{ ~~if } s\neq 0.
		\end{cases}
	\end{align*}
One can show easily (c.f. \cite{XWY14}, \cite{XWY16} and \cite{XCT22})
	\begin{align}\label{rhoproperty}
		\rho_s(\omega_1,\ldots,\omega_n)=\rho_{s-1}(\omega_1,\ldots,\omega_{n-2})+\omega_n\rho_s(\omega_1,\ldots,\omega_{n-1}),
	\end{align}
	and, for $0\leq j\leq n$,
   \begin{align}
      \rho_{n-j}(\omega_1,\ldots,\omega_{2n})=\sum_{1\leq s_1\leq t_1<s_2\leq t_2<\cdots < s_j\leq t_j\leq n}
         \prod_{k=1}^j\omega_{2s_k-1}\omega_{2t_k},\label{0b0}
   \end{align}
   \begin{align}
   &\rho_{n-j}(\omega_1,\ldots,\omega_{2n+1})=\sum_{s=j}^n\rho_{s-j}(\omega_1,\ldots,\omega_{2s})\omega_{2s+1}\nonumber\\
   =&\sum_{1\leq s_0\leq t_1< s_1\leq t_2<s_2\leq \cdots \leq t_j< s_j\leq n+1}
         \omega_{2s_0-1}\prod_{k=1}^j\omega_{2t_k}\omega_{2s_k-1}.\label{0b1}
   \end{align}
For $a,b\in\mathbb{F}_q^*$, let $\omega'_{2s-1}=a\omega_{2s-1}$ and $\omega'_{2s}=b\omega_{2s}$ for $s=1,2,\ldots$, then from (\ref{0b0}) and (\ref{0b1}) we see easily
\begin{gather}\label{0b3}
\rho_{n-j}(\omega'_1,\ldots,\omega'_{2n})=a^jb^j\rho_{n-j}(\omega_1,\ldots,\omega_{2n}),\\
\rho_{n-j}(\omega'_1,\ldots,\omega'_{2n+1})=a^{j+1}b^j\rho_{n-j}(\omega_1,\ldots,\omega_{2n+1}).\label{0b4}
\end{gather}
	
	Since the graph $\Lambda_{k,q}$ is edge-transitive, without loss of generality, we will concentrate on the backtrackless walks which are leading by the two vertices expressed by the all-zero vector.
Let $\Gamma=[l^{(1)}]\langle r^{(1)}\rangle[l^{(2)}]\langle r^{(2)}\rangle\cdots$ be such a given backtrackless walk of $\Lambda_{k,q}$, where $[l^{(1)}]=(0,0,\ldots,0)$ and $\langle r^{(1)}\rangle=(0,0,\ldots,0)$. For $i\geq1$, let $x_i$ and $y_i$ denote the first entries (or colors) of $[l^{(i)}]$ and $\langle r^{(i)}\rangle$, respectively, and write
	\begin{align}\label{relationxy}
		u_i=x_{i+1}-x_{i}, ~~~~~~~v_i=y_{i+1}-y_{i}.
	\end{align}
	Clearly, we have $u_i\neq0$ and $v_i\neq0$. As a refinement of a closed-form expression given in \cite{XWY14} for the backtrackless walks leading by $[l^{(1)}]=(0,0,\ldots,0)$, the following lemma was shown in \cite{XCT22}.
	\begin{lemma}\label{path}
		For any $i\geq 1$ and $j\geq  0$, let $l_j^{(i+1)}$ denote the $(j+1)$-th entries of $[l^{(i+1)}]$. Then, we have
		\begin{align}
			l_{4j}^{(i+1)}&=\rho_{i-j-1}(u_1,v_1,\ldots,u_{i-1},v_{i-1},u_i),\label{4j}\\
			l_{4j+1}^{(i+1)}&=\rho_{i-j-2}(v_1,u_2,\ldots,v_{i-1},u_i),\label{4j+1}\\
			l_{4j+2}^{(i+1)}&=y_{i+1}l_{4j}^{(i+1)}-\rho_{i-j-1}(u_1,v_1,\ldots,u_i,v_i),\label{4j+2}\\
			l_{4j+3}^{(i+1)}&=y_{i+1}l_{4j+1}^{(i+1)}-\rho_{i-j-2}(v_1,u_2\ldots,v_{i-1},u_i,v_i).\label{4j+3}
		\end{align}
	\end{lemma}

This lemma shows a compact expression for the vertices on the walk $\Gamma$.	
     For convenience, we say the walk $\Gamma$ is of type $(u_1,v_1,u_2,v_2,\ldots)$. If the first $2i$ vertices in the walk $\Gamma$ form a circuit in $\Lambda_{k,q}$ of length $2i$, we also call it a backtrackless circuit of type $(u_1,v_1,\ldots, u_i,v_i)$.

     We note that
     \begin{align}\label{4j-1}
       \rho_{i-1}(v_1,u_2\ldots,v_{i-1},u_i,v_i)=v_1+\cdots+v_i=y_{i+1}=y_1=0
     \end{align}
is always a necessary condition for $\Lambda_{k,q}$ to have a backtrackless circuit of type $(u_1,v_1,\ldots, u_i,v_i)$.

\section{Backtrackless Circuits in $\Lambda_{4t+2,q}$}	

In this section we show an identity on the homogeneous polynomials introduced in Section~2 at first. By using this identity, we show then that each backtrackless circuit in $\Lambda_{4t+1,q}$ ensures the existence of a backtrackless circuit of the same type in $\Lambda_{4t+2,q}$ and deduce
$g(\Lambda_{4t+2,q})=g(\Lambda_{4t+1,q})$ for any $t\geq 1$ in final.
\begin{lemma}\label{lem4j+1}
For any integers $n,t$ with $n\geq 1$, let
\begin{align}
  \Delta^n_{2t-1}=& \rho_{n-t}(v_1,u_2,\ldots,v_{n-1},u_n,v_n), \label{delta01}\\
  \nabla^n_{2t-1}=& \rho_{n-t}(u_1,v_1,\ldots,u_{n-1},v_{n-1},u_n), \label{delta02}\\
  \Delta^n_{2t}=& \rho_{n-1-t}(v_1,u_2,\ldots,v_{n-1},u_n), \label{delta03}\\
  \nabla^n_{2t}=& \rho_{n-t}(u_1,v_1,\ldots,u_n,v_n). \label{delta04}
\end{align}
Then, for $n\geq 1$ we have
\begin{align}\label{girth4j+1}
\sum_{s}(-1)^{s}\nabla^n_{s}\Delta^n_{2j-s}=0,\text{ for }j\geq 1.
\end{align}
\end{lemma}
\begin{proof}
From $\Delta^1_0=\nabla^1_0=1$, $\Delta^1_1=v_1$, $\nabla^1_1=u_1$, $\nabla^1_2=u_1v_1$, $\Delta^1_s=\nabla^1_t=0$ for any $s\not\in\{0,1\}$
and $t\not\in\{0,1,2\}$, it can be checked easily that (\ref{girth4j+1}) is valid for $n=1$.

To show that (\ref{girth4j+1}) is valid for $n>1$, we note that for any integer $t$ according to (\ref{rhoproperty}) we have
\begin{align*}
  \Delta^n_{2t-1}
  =&\rho_{n-t-1}(v_1,u_2,\ldots,v_{n-2},u_{n-1},v_{n-1})\\
  &+v_n\rho_{n-t-1}(v_1,u_2,\ldots,v_{n-2},u_{n-1})\\
  &+v_nu_n\rho_{n-t}(v_1,u_2,\ldots,v_{n-2},u_{n-1},v_{n-1})\\
=&\Delta^{n-1}_{2t-1}+v_n\Delta^{n-1}_{2t-2}+v_nu_n\Delta^{n-1}_{2t-3},\\
\nabla^n_{2t-1}=&\rho_{n-t-1}(u_1,v_1,\ldots,u_{n-2},v_{n-2},u_{n-1})\\
&+u_n\rho_{n-t}(u_1,v_1,\ldots,u_{n-1},v_{n-1})\\
=&\nabla^{n-1}_{2t-1}+u_n\nabla^{n-1}_{2t-2},\\
\Delta^n_{2t}=&\rho_{n-2-t}(v_1,u_2,\ldots,v_{n-2},u_{n-1})\\
&+u_n\rho_{n-1-t}(v_1,u_2,\ldots,v_{n-2},u_{n-1},v_{n-1})\\
=&\Delta^{n-1}_{2t}+u_n\Delta^{n-1}_{2t-1},\\
\nabla^n_{2t}=& \rho_{n-t-1}(u_1,v_1,\ldots,u_{n-1},v_{n-1}),\\
&+v_n\rho_{n-t-1}(u_1,v_1,\ldots,u_{n-2},v_{n-2},u_{n-1})\\
&+v_nu_n\rho_{n-t}(u_1,v_1,\ldots,u_{n-1},v_{n-1})\\
=&\nabla^{n-1}_{2t}+v_n\nabla^{n-1}_{2t-1}+v_nu_n\nabla^{n-1}_{2t-2},
\end{align*}
then we get
\begin{align*}
&\sum_{s}(-1)^{s}\nabla^n_{s}\Delta^n_{2j-s}\\
=&\sum_t(\nabla^{n-1}_{2t}+v_n\nabla^{n-1}_{2t-1}+v_nu_n\nabla^{n-1}_{2t-2})(\Delta^{n-1}_{2j-2t}+u_n\Delta^{n-1}_{2j-2t-1})\\
&-\sum_t(\nabla^{n-1}_{2t-1}+u_n\nabla^{n-1}_{2t-2})(\Delta^{n-1}_{2j-2t+1}+v_n\Delta^{n-1}_{2j-2t}+v_nu_n\Delta^{n-1}_{2j-2t-1})\\
=&\sum_t(\nabla^{n-1}_{2t}\Delta^{n-1}_{2j-2t}-\nabla^{n-1}_{2t-1}\Delta^{n-1}_{2j-2t+1})\\
&+u_n\sum_t(\nabla^{n-1}_{2t}\Delta^{n-1}_{2j-2t-1}-\nabla^{n-1}_{2t-2}\Delta^{n-1}_{2j-2t+1})\\
=&\sum_s(-1)^{s}\nabla^{n-1}_{s}\Delta^{n-1}_{2j-s}.
\end{align*}
Therefore, one can show easily by induction that (\ref{girth4j+1}) is valid for any positive integer $n$.
\end{proof}

The following theorem is then a simple corollary of Lemmas~\ref{path} and \ref{lem4j+1}.
\begin{theorem}\label{4t+1&4t+2}
For $t\geq 1$, $\Lambda_{4t+2,q}$ has a backtrackless circuit of type $(u_1,v_1,\ldots,$ $u_i,v_i)$ if and only if $\Lambda_{4t+1,q}$ has a backtrackless circuit of the same type.
In particular, we have $g(\Lambda_{4t+2,q})=g(\Lambda_{4t+1,q})$ for $t\geq 1$.
\end{theorem}
\begin{proof}
     	Assume that there is a backtrackless circuit of type $(u_1,v_1,\ldots,u_i,v_i)$ in $\Lambda_{4t+1,q}$, that is, $$v_1+\cdots+v_i=y_{i+1}=y_1=0$$ and
     $l^{(i+1)}_k=0$ for $0\leq k\leq 4t+1$.
      According to Lemma~\ref{path}, by using the notations defined in Lemma~\ref{lem4j+1} we have $\Delta^i_{2t+2}=0$ and $\Delta^i_k=\nabla^i_k=0$ for $1\leq k\leq 2t+1$.
     	Therefore, from $\Delta^i_0=1$ and Lemma \ref{lem4j+1} we see
     $$l_{4t+2}^{(i+1)}=y_{i+1}l_{4t}^{(i+1)}-\nabla^i_{2t+2}=\sum_{s=0}^{2t+1}(-1)^{s}\nabla^i_s\Delta^i_{2t+2-s}=0,$$
     	and then $\Lambda_{4t+2,q}$ also has a backtrackless circuit of type $(u_1,v_1,\ldots,u_i,v_i)$. On the other hand, we note that
     $\Lambda_{4t+2,q}$ has a backtrackless circuit of type $(u_1,v_1,\ldots,u_i,v_i)$ implies naturally that $\Lambda_{4t+1,q}$ has a backtrackless circuit of the same type. The proof is complete.
\end{proof}

\section{Backtrackless Circuits in $\Lambda_{4s+3,q}$}	

All the arguments given in this section will be based on the existence of backtrackless circuits of type $(u_1,v_1,\ldots,u_{2n},v_{2n})$ with
     	\begin{align}\label{abs_v_eq1}
     		v_{2j-1}=-v_{2j}=1,\,\, j=1,2,\ldots,n,
     	\end{align}
in the graph $\Lambda_{k,q}$. To show the existence of such circuits, we deduce some equalities on the homogeneous polynomials $\rho_s(\cdot,\ldots,\cdot)$ at first.
     \begin{lemma}\label{2n&n-cycle}
For any integer $t$ and tuple $(u_1,v_1,\ldots,u_{2n},v_{2n})$ over $\mathbb{F}_q^{\ast}$ with (\ref{abs_v_eq1}), we have
        \begin{align}
        	&\rho_{2n-2t}(u_1,v_1,\ldots,u_{2n-1},v_{2n-1},u_{2n})\nonumber\\
        	=&(-1)^{t-1}\rho_{n-t}(u_1,\ldots,u_{2n})+(-1)^{t}\rho_{n-t-1}(u_2,\ldots,u_{2n-1}),\label{rho_2n-2t}\\
        	&\rho_{2n+1-2t}(u_1,v_1,\ldots,u_{2n-1},v_{2n-1},u_{2n})\nonumber\\
        	=&(-1)^{t-1}\rho_{n-t}(u_2,\ldots,u_{2n})+(-1)^{t-1}\rho_{n-t}(u_1,\ldots,u_{2n-1}),\\
        	&\rho_{2n-1-2t}(v_1,u_2,\ldots,v_{2n-1},u_{2n})=(-1)^{t}\rho_{n-t-1}(u_2,\ldots,u_{2n-1}),\\
        	&\rho_{2n-2t}(v_1,u_2,\ldots,v_{2n-1},u_{2n})=(-1)^{t-1}\rho_{n-t}(u_2,\ldots,u_{2n}),\\
        	&\rho_{2n-2t}(u_1,v_1,\ldots,u_{2n},v_{2n})=(-1)^{t}\rho_{n-t}(u_1,\ldots,u_{2n}),\\
        	&\rho_{2n+1-2t}(u_1,v_1,\ldots,u_{2n},v_{2n})=(-1)^{t}\rho_{n-t}(u_2,\ldots,u_{2n}),\\
        	&\rho_{2n-2t}(v_1,u_2,\ldots,v_{2n-1},u_{2n},v_{2n})=(-1)^{t}\rho_{n-t}(u_2,\ldots,u_{2n}),\\
        	&\rho_{2n+1-2t}(v_1,u_2,\ldots,v_{2n-1},u_{2n},v_{2n})=0.
        \end{align}
     \end{lemma}
     \begin{proof}
     	We give a proof only for the equality (\ref{rho_2n-2t}). The others can be proved similarly.
     	
     	It is obvious that (\ref{rho_2n-2t}) is valid if $t\leq 0$ or $t\geq n$. Since for $1\leq s\leq  n$ we have $y_{2s-1}=0$ and $y_{2s}=1$, according to (\ref{0b0}) and (\ref{0b1}), for $1\leq t<n$, we have
     	\begin{align*}
     		&\rho_{2n-2t}(u_1,v_1,\ldots,u_{2n-1},v_{2n-1},u_{2n})\\
     		=&\sum_{1\leq j_1\leq j_2<j_3\leq\cdots\leq j_{4t-2}<j_{4t-1}\leq 2n}u_{j_1}\prod_{s=1}^{2t-1}v_{j_{2s}}u_{j_{2s+1}}\\
     		=&\sum_{1\leq j_1<j_3<\cdots<j_{4t-1}\leq 2n}u_{j_1}\prod_{s=1}^{2t-1}u_{j_{2s+1}}\sum_{j_{2s-1}\leq l<j_{2s+1}} v_{j_{2s}}\\
     		=&\sum_{1\leq j_1<j_3<\cdots<j_{4t-1}\leq 2n}u_{j_1}\prod_{s=1}^{2t-1}u_{j_{2s+1}}(y_{j_{2s+1}}-y_{j_{2s-1}})\\
     		=&(-1)^{t-1}\sum_{1\leq i_1\leq i_2<i_3\leq i_4<\cdots<i_{2t-1}\leq i_{2t}\leq n}\prod_{s=1}^t u_{2i_{2s-1}-1}u_{2i_{2s}}\\
     		&\qquad +(-1)^{t}\sum_{1\leq i_1<i_2\leq i_3<i_4\leq\cdots\leq i_{2t-1}<i_{2t}\leq n}\prod_{s=1}^t u_{2i_{2s-1}}u_{2i_{2s}-1}\\
     		=&(-1)^{t-1}\rho_{n-t}(u_1,\ldots,u_{2n})+(-1)^{t}\rho_{n-t-1}(u_2,\ldots,u_{2n-1}),
     	\end{align*}
        i.e. (\ref{rho_2n-2t}) is valid for $1\leq t<n$.
     \end{proof}

The following lemma shows that, from any backtrackless circuit of length $2n$ in $\Lambda_{2s,q}$, one can construct a backtrackless circuit of length $4n$ in $\Lambda_{4s+3,q}$, for any $s\geq 1$ and $n\geq 3$.
     \begin{lemma}\label{cycle_2s&4s+3}
     	Assume that $s\geq 1$ and $n\geq 3$. The graph $\Lambda_{4s+3,q}$ has a backtrackless circuit of type $(u_1,v_1,\ldots,u_{2n},v_{2n})$ with (\ref{abs_v_eq1}) if and only if $\Lambda_{2s,q}$ has a backtrackless circuit of length $2n$.
     \end{lemma}
     \begin{proof}
     	Assume that $v_1,\ldots,v_{2n}\in\mathbb{F}_q^{\ast}$ satisfy (\ref{abs_v_eq1}).
     	According to Lemma~\ref{path}, the graph $\Lambda_{4s+3,q}$ has a backtrackless circuit of type $(u_1,v_1,\ldots,u_{2n},v_{2n})$ if and only if
     	\begin{align}
     		\begin{cases}\label{2n-cycle}
     			&\rho_{2n-j-2}(v_1,u_2,\ldots,v_{2n-1},u_{2n},v_{2n})=0,\\
     			&\rho_{2n-j-1}(u_1,v_1,\ldots,u_{2n},v_{2n})=0,\\
     			&\rho_{2n-j-2}(v_1,u_2,\ldots,v_{2n-1},u_{2n})=0,\\
     			&\rho_{2n-j-1}(u_1,v_1,\ldots,u_{2n-1},v_{2n-1},u_{2n})=0,
     		\end{cases}
     	    \text{ ~~for } 0\leq j\leq s.
     	\end{align}

     	If $s=2w$ is even, according to Lemma~\ref{2n&n-cycle} we see that (\ref{2n-cycle}) is equivalent to
     $\rho_{n-1}(u_1,\ldots,u_{2n-1})=\rho_{n-1}(u_2,\ldots,u_{2n})=0$ and
     	\begin{align}\label{s=2w}
     		\begin{cases}
     	  	   &\rho_{n-j-1}(u_1,\ldots,u_{2n-1})=0,\\
     		   &\rho_{n-j-1}(u_2,\ldots,u_{2n})=0,\\
     		   &\rho_{n-j}(u_1,\ldots,u_{2n})=0,\\
     		   &\rho_{n-j-1}(u_2,\ldots,u_{2n-1})=0,
     	   \end{cases}
           \text{ ~~~~for } 1\leq j\leq w,
     	\end{align}
        that is, the graph $\Lambda_{4w,q}$ has a backtrackless circuit of type $(u_1,\ldots,u_{2n})$ on account to Lemma~\ref{path}.

        If $s=2w-1$ is odd, according to Lemma~\ref{2n&n-cycle} we see that (\ref{2n-cycle}) is equivalent to
        \begin{align}
        	\begin{cases}\label{s=2w-1}
        		&\rho_{n-j}(u_1,\ldots,u_{2n})=0,\\
        		&\rho_{n-j-1}(u_2,\ldots,u_{2n-1})=0,\\
        		&\rho_{n-j}(u_1,\ldots,u_{2n-1})=0,\\
        		&\rho_{n-j}(u_2,\ldots,u_{2n})=0,
        	\end{cases}
        	\text{ ~~~~for } 1\leq j\leq w.
        \end{align}
        that is, the graph $\Lambda_{4w-2,q}$ has a backtrackless circuit of type $(u_1,\ldots,u_{2n})$ on account to Lemma~\ref{path}.
     \end{proof}

Based on the existence of backtrackless circuits of type $(u_1,v_1,\ldots,u_{2n},v_{2n})$ with
     	(\ref{abs_v_eq1}), on the girth of $\Lambda_{k,q}$ one can show the following theorem by using Lemmas~\ref{path}, \ref{2n&n-cycle} and \ref{cycle_2s&4s+3}.

     \begin{theorem}\label{relation2s&4s+3(4)}
     	Assume $n\geq 3$ and $s,w\geq 1$.
     	\begin{enumerate}
     		\item If $g(\Lambda_{2s,q})\leq 2n$, then $g(\Lambda_{4s+3,q})\leq 4n$.
     		\item If $g(\Lambda_{2n-4,q})=2n$, then $g(\Lambda_{4n-5,q})=4n$.
     	    \item If $g(\Lambda_{4w-2,q})=4w+2$, then $g(\Lambda_{8w,q})=8w+4$.
            \item If $q$ is a power of 2 and $g(\Lambda_{4w,q})\leq 2n$, then $g(\Lambda_{8w+4,q})\leq 4n$.
      	\end{enumerate}
     \end{theorem}
     \begin{proof}
     	The first statement follows simply from Lemma \ref{cycle_2s&4s+3}. Furthermore, the second statement follows immediately on account to $g(\Lambda_{4n-5,q})\geq 4n$.
     	
     	To show the third statement, we set $n=2w+1$ and assume that the graph $
     	\Lambda_{4w-2,q}$ has a backtrackless circuit of type $(u_1,\ldots,u_{2n})$. Then, according to Lemma~\ref{path} we have (\ref{s=2w-1}). If $\rho_{n-w-1}(u_2,\ldots,u_{2n})=0$, then $\Lambda_{4w-1,q}$ has a backtrackless circuit of type $(u_1,\ldots,u_{2n})$, contradicts  $g(\Lambda_{4w-1,q})\geq 4w+4>2n$.
     	If $\rho_{n-w-1}(u_1,\ldots,u_{2n-1})=0$, then $\Lambda_{4w-1,q}$ has a backtrackless circuit of type $(u_{2n},\ldots,u_1)$, contradicts  $g(\Lambda_{4w-1,q})\geq 4w+4>2n$ too. Hence, we have $\rho_{n-w-1}(u_2,\ldots,u_{2n})\rho_{n-w-1}(u_1,\ldots,u_{2n-1})\neq 0$. Let
     	$$\alpha=-\rho_{n-w-1}(u_2,\ldots,u_{2n})/\rho_{n-w-1}(u_1,\ldots,u_{2n-1}).$$
     	We multiply the entries with odd indices in the tuple $(u_1,\ldots,u_{2n})$ by $\alpha$, and denote the resulting tuple by the same notation. Then, according to (\ref{0b3}) and (\ref{0b4}) one can check easily that the modified tuple $(u_1,\ldots,u_{2n})$ satisfies (\ref{s=2w-1}) and
     	$$\rho_{n-w-1}(u_2,\ldots,u_{2n})+\rho_{n-w-1}(u_1,\ldots,u_{2n-1})=0.$$
     	Hence, according to Lemma \ref{2n&n-cycle} we have
     	$$\rho_{2n-2w-1}(u_1,v_1,\ldots,u_{2n-1},v_{2n-1},u_{2n})=0$$
     	and (\ref{2n-cycle}) with $s=2w-1$, where the tuple $(v_1,\ldots,v_{2n})$ satisfies (\ref{abs_v_eq1}). Therefore, according to Lemma~\ref{path} we see $\Lambda_{8w,q}$ has a backtrackless circuit of type $(u_1,v_1,\ldots,u_{2n},v_{2n})$ and thus we have $g(\Lambda_{8w,q})\leq 4n=8w+4$. Hence, from $g(\Lambda_{8w,q})\geq 8w+4$ we see $g(\Lambda_{8w,q})=8w+4$.

     To show the last statement, we assume that $q$ is a power of 2 and that the graph $\Lambda_{4w,q}$ has a backtrackless circuit of type $(u_1,\ldots,u_{2n})$.
     Then, according to Lemma~\ref{path} we have
     $$\rho_{n-1}(u_1,\ldots,u_{2n-1})=\rho_{n-1}(u_2,\ldots,u_{2n})=0$$
     and (\ref{s=2w}), that is,
     $\nabla^n_{t}=\Delta^n_{t}=0$ holds for $1\leq t\leq 2w+1$ when we modify accordingly the definition of the notations $\nabla^n_{t}$, $\Delta^n_{t}$.
     Therefore, according to Lemma~\ref{lem4j+1} and that the characteristic of $\mathbb{F}_q$ is 2, we see
     \begin{align*}
     &\rho_{n-w-2}(u_2,\ldots,u_{2n-1})-\rho_{n-w-1}(u_1,\ldots,u_{2n})\\
     =&\Delta^n_{2w+2}-\nabla^n_{2w+2}\\
     =&2\Delta^n_{2w+2}+\sum_{s=1}^{2w+1}(-1)^{s}\nabla^n_{s}\Delta^n_{2w+2-s}=0.
     \end{align*}
     Hence, according to Lemma \ref{2n&n-cycle} we have
     $$
     \rho_{2n-2w-2}(u_1,v_1,\ldots,u_{2n-1},v_{2n-1},u_{2n})=0
     $$
     	and (\ref{2n-cycle}) with $s=2w$, where the tuple $(v_1,\ldots,v_{2n})$ satisfies (\ref{abs_v_eq1}). Therefore, according to Lemma~\ref{path} we see $\Lambda_{8w+4,q}$ has a backtrackless circuit of type $(u_1,v_1,\ldots,u_{2n},v_{2n})$ and thus we have $g(\Lambda_{8w+4,q})\leq 4n$.
     \end{proof}

     \begin{exam}
        \begin{itemize}
        \item For $k\geq 2$, the girth of $\Lambda_{k,2}$ has been determined \cite{XWY14}: $g(\Lambda_{k,2})=2^{s}$, where $s$ is the integer with $2^{s-1}-4<k\leq 2^{s}-4$.
        	\item Suppose $q\geq3$. According to Theorem~\ref{relation2s&4s+3(4)}, from $g(\Lambda_{2,q})=6$ \cite{XWY14} we see $g(\Lambda_{7,q})=g(\Lambda_{8,q})=12$, $g(\Lambda_{18,q})\leq g(\Lambda_{19,q})=24$, $g(\Lambda_{38,q})\leq g(\Lambda_{39,q})\leq 48$, $g(\Lambda_{78,q})\leq g(\Lambda_{79,q})\leq 96$ and $g(\Lambda_{159,q})\leq 192$.
        	\item Suppose $q>3$. According to Theorem~\ref{relation2s&4s+3(4)}, from $g(\Lambda_{4,q})=8$ \cite{XCT22} we see $g(\Lambda_{11,q})=16$.
        	According to Theorem~\ref{4t+1&4t+2}, from $g(\Lambda_{5,q})=10$ \cite{XCT22} we see $g(\Lambda_{6,q})=10$ and then, according to Theorem~\ref{relation2s&4s+3(4)}, we have
        $g(\Lambda_{15,q})=g(\Lambda_{16,q})=20$ and $g(\Lambda_{35,q})=40$.
        \end{itemize}
     \end{exam}

When the characteristic of $\mathbb{F}_q$ is 2, one can deduce further the following corollary easily.
\begin{corollary}\label{cor00}
Assume $g(\Lambda_{2s,q})=2s+4$, where $q$ is a power of 2 and $s\geq 1$. Then, for any $t\geq 1$ we have
     \begin{align*}
         g(\Lambda_{2^{t}(s+2)-4,q})=g(\Lambda_{2^{t}(s+2)-5,q})=2^{t}(s+2),
     \end{align*}
     where $\Lambda_{1,q}$ is defined as a graph isomorphic to $\Lambda_{2,q}$ for convenience.
\end{corollary}
\begin{proof}
From $g(\Lambda_{k,q})\geq k+4$ and Theorem~\ref{relation2s&4s+3(4)}, we see easily that $g(\Lambda_{2^{t}(s+2)-4,q})$ $=2^{t}(s+2)$ is valid for any $t\geq 1$.
Furthermore, $g(\Lambda_{2^{t}(s+2)-5,q})=2^{t}(s+2)$ follows from $2^{t}(s+2)\leq g(\Lambda_{2^{t}(s+2)-5,q})\leq g(\Lambda_{2^{t}(s+2)-4,q})$.
\end{proof}

\begin{exam}
Assume that $q\geq 4$ is a power of 2. According to Corollary~\ref{cor00}, we have the following three statements.
\begin{itemize}
  \item From $g(\Lambda_{2,q})=6$, we see $g(\Lambda_{2^{t}3-4,q})=g(\Lambda_{2^{t}3-5,q})=2^{t}3$ for $t\geq 1$.
  \item From $g(\Lambda_{4,q})=8$, we see $g(\Lambda_{2^{t+2}-4,q})=g(\Lambda_{2^{t+2}-5,q})=2^{t+2}$ for $t\geq 1$.
  \item From $g(\Lambda_{6,q})=10$, we see $g(\Lambda_{2^{t}5-4,q})=g(\Lambda_{2^{t}5-5,q})=2^{t}5$ for $t\geq 1$.
\end{itemize}
\end{exam}

     For prime $p$, we write $m|_pn$ if $m|(np^r)$ for some $r\geq 0$. The following lemma is from \cite{XWY16}.
     \begin{lemma}\label{g(Lambdakq)=k+5}
     	For $q=p^s$ and $t\geq 1$ with $(t+2)|_p(q-1)$,
     \begin{align}\label{v0}
       g(\Lambda_{2t-1,q})=g(\Lambda_{2t,q})=2t+4.
     \end{align}
     \end{lemma}
     	
     The following theorem follows simply from Theorem~\ref{relation2s&4s+3(4)} and Lemma~\ref{g(Lambdakq)=k+5}.

     \begin{theorem}\label{theo50}
     Assume that $q$ is a power of odd prime $p$ and $s,t$ are positive integers with $(2t-1)|_p(q-1)$ and $2^s\|(q-1)$. Then, Conjecture~1 is valid when
$k=2^{s+2}(2t-1)-5$, i.e.
\begin{align}\label{new1}
 g(\Lambda_{2^{s+2}(2t-1)-5,q})=2^{s+2}(2t-1).
\end{align}
\end{theorem}
     \begin{proof}
     Clearly, we have $2^{s}(2t-1)|_p(q-1)$. Hence, from Lemma~\ref{g(Lambdakq)=k+5} we see
     $$g(\Lambda_{2^{s+1}(2t-1)-5,q})=g(\Lambda_{2^{s+1}(2t-1)-4,q})=2^{s+1}(2t-1),$$
     therefore, from Theorem~\ref{relation2s&4s+3(4)} we see (\ref{new1}).
\end{proof}

We note that the result shown in this theorem is not included by Lemma~\ref{g(Lambdakq)=k+5}.
At the end of this section, we investigate the girth of $\Lambda_{k,3}$ for small $k$.

     \begin{exam}
     	\begin{itemize}
     		\item The positive integer $t$'s satisfying $(t+2)|_3(3-1)$ are $1,4,7,16,25,52,79,160,\ldots$. Then, according to Lemma~\ref{g(Lambdakq)=k+5} we have $g(\Lambda_{2,3})=6$, $g(\Lambda_{7,3})$ $=g(\Lambda_{8,3})=12$, $g(\Lambda_{13,3})=g(\Lambda_{14,3})=18$, $g(\Lambda_{31,3})=g(\Lambda_{32,3})=36$, $g(\Lambda_{49,3})=g(\Lambda_{50,3})=54$, $g(\Lambda_{103,3})=g(\Lambda_{104,3})=108$, $g(\Lambda_{157,3})=g(\Lambda_{158,3})=162$, $g(\Lambda_{319,3})=g(\Lambda_{320,3})=324$, $\ldots$.
     		\item From $g(\Lambda_{32,3})=36$, according to Theorem~\ref{relation2s&4s+3(4)} we see $g(\Lambda_{66,3})\leq g(\Lambda_{67,3})$ $=72$, $g(\Lambda_{134,3})\leq g(\Lambda_{135,3})\leq144$ and $g(\Lambda_{271,3})\leq288$.
     	    \item From $g(\Lambda_{104,3})=108$, according to Theorem~\ref{relation2s&4s+3(4)} we see $g(\Lambda_{211,3})=216$.
     	\end{itemize}
     \end{exam}

     The known results on the girth of $\Lambda_{k,3}$ for $2\leq k \leq 320$ are summarized in the following table.

     \begin{table}[h]\caption{Girth of $\Lambda_{k,3}$ for $2\leq k \leq 320$.}
     	\centering
     	\begin{tabular}{|c|c|c|c|c|c|c|c|c|c|}
     		\hline
     		2  & 3   & 4         & 5         & 6   & 7          & 8   & 13   & 14        & 19* \\
     		6  & 8   & 12        & 12        & 12  & 12         & 12  & 18   & 18         & 24  \\
     		\hline\hline
     		31 & 32 & 39*      & 49 & 50 & 67* & 79*       & 103   & 104 & 135*        \\
     		36 & 36  &$\leq 48$ & 54 & 54  & 72  & $\leq 96$ & 108   & 108  & $\leq 144$ \\
      		\hline\hline
      	    157  & 158 & 159*       & 211*       & 271*       & 319  & 320 &  &  & \\
      		162  & 162  & $\leq 192$ & 216        & $\leq 288$ & 324  & 324 &  &  & \\
      		\hline
     	\end{tabular}
     \end{table}
\noindent In this table, the mark * indicates the exact values or upper bounds of $g(\Lambda_{k,3})$ are obtained by the methods proposed in this paper. We note that the girth cycles of $\Lambda_{k,3}$ were determined in \cite{XCT22} for $3\leq k\leq 8$. In particular,
the results $g(\Lambda_{3,3})=8$ and $g(\Lambda_{k,3})=12$ for $4\leq k\leq 8$ can be found therein.

\section{Upper Bound of $g(\Lambda_{k,q})$}

    In this section, we manage to deduce an upper bound for the girth of $\Lambda_{k,q}$ for $q\geq 3$.

    Assume that $q\geq 3$ is a given prime power and the number of positive factors of $q-1$ is $n$.
    Let $k_1,k_2,\ldots$ be the odd integers in ascending order with $\frac{k_i+5}{2}|_p(q-1)$, where $p$ is the characteristic of $\mathbb{F}_q$.
    Let $i_0$ be the integer with $k_{i_0}=2q-5$.

    \begin{lemma}\label{lem-ub}
    For any $i\geq i_0$
    \begin{align}\label{ub00}
    k_{i+n}=pk_{i}+5p-5.
    \end{align}
    \end{lemma}
\begin{proof}
Suppose $q=p^m$.
For $0\leq j\leq m-1$, let $d_{j,1},\ldots,d_{j,t_j}$ denote the different factors of $q-1$ with $p^j\leq d_{j,t}<p^{j+1}$ for $t=1,\ldots,t_j$.
Then, we have $\sum_{j=0}^{m-1}t_j=n$ and for any $s\geq 0$ from $(p,d_{j,t})=1$ we see
\begin{align*}
  \left\{k_{i_0+sn+i}|0\leq i<n\right\}=\bigcup_{0\leq j< m}\left\{\left.2d_{j,t}p^{s+m-j}-5\right|1\leq t\leq t_j\right\},
\end{align*}
which implies (\ref{ub00}).
\end{proof}

Let
\begin{align}\label{ub01'}
  T_q=\max_{i\geq i_0}\frac{k_{i+1}+5}{k_{i}+5}.
\end{align}
From Lemma~\ref{lem-ub} we see that $T_q$ can also be given by
\begin{align}\label{ub01}
  T_q=\max_{i_0\leq i< i_0+n}\frac{k_{i+1}+5}{k_{i}+5}.
\end{align}
Clearly, $1<T_q<p$, and for any $i\geq i_0$ we have $\frac{k_{i+1}+5}{2}\leq T_q\frac{k_i+5}{2}$, i.e.
\begin{align}\label{ub02}
  k_{i+1}\leq T_qk_i+5T_q-5.
\end{align}
\begin{exam}
If $q=5^2$, then the positive factors of $5^2-1=24$ are $1,2,3,4,6,$ $8,12,24$ and the positive integers $t$ with $t|_524$ are
\begin{align*}
  1,2,3,4,5,6,8,10,12,15,20,24,25,30,40,50,60,75,100,120,125,\ldots
\end{align*}
Hence, $T_q=\max\{6/5,8/6,10/8,12/10,15/12,20/15,24/20,25/24\}=4/3$.
\end{exam}

From Lemma~\ref{g(Lambdakq)=k+5} we see $g(\Lambda_{k_i-1,q})\leq g(\Lambda_{k_i,q})= k_i+5$, hence from Theorem~\ref{relation2s&4s+3(4)} we see
$$g(\Lambda_{2(k_i-1)+2,q})\leq g(\Lambda_{2(k_i-1)+3,q})\leq 2(k_i+5)$$
    and by induction we have
    \begin{align}\label{ub05}
     g(\Lambda_{2^s(k_i+1)-2,q})\leq g(\Lambda_{2^s(k_i+1)-1,q})\leq 2^s(k_i+5),\text{ for any }s\geq 0.
    \end{align}

    \begin{theorem}
Let $q$ be a prime power.
\begin{enumerate}
  \item If $T_q\leq 2$, then for $k\geq q$ we have
\begin{align}
\label{ub10}
  g(\Lambda_{k,q})\leq T_q(k+4).
\end{align}
  \item If $T_q>2$ and $k\geq \max\{q,8T_q^2-10T_q-3\}$, then we have
    \begin{align}
    \label{ub11}
      g(\Lambda_{k,q})\leq 2k+4T_q+1.
    \end{align}
\end{enumerate}
\end{theorem}
     \begin{proof}
Without loss of generality,  we assume $k_i<k< k_{i+1}$ for some $i\geq i_0$.

If $T_q\leq 2$, then from (\ref{ub02}) and Lemma~\ref{g(Lambdakq)=k+5} we see
\begin{align*}
  g(\Lambda_{k,q})\leq g(\Lambda_{k_{i+1},q})=k_{i+1}+5\leq T_q(k_i+5)\leq T_q(k+4),
\end{align*}
i.e. (\ref{ub10}) is valid.

Now we assume $T_q>2$ and $k\geq 8T_q^2-10T_q-3$.

    If $\frac{k_{i+1}}{2}< k\leq k_{i+1}$, then from Lemma~\ref{g(Lambdakq)=k+5} we have
    \begin{align}\label{ub12}
      g(\Lambda_{k,q})\leq g(\Lambda_{k_{i+1},q})=k_{i+1}+5\leq 2k+4.
    \end{align}

    If $k_i<k<\frac{k_{i+1}}{2}$, for the integer $s$ with
    \begin{align}\label{ub13}
     2^s(k_i+1)\leq k< 2^{s+1}(k_i+1)
    \end{align}
    from (\ref{ub02}) we see
    \begin{align*}
    2^s(k_i+1)\leq k\leq \frac{k_{i+1}-1}{2}\leq \frac{T_qk_i+5T_q-6}{2},
    \end{align*}
    and then from (\ref{ub05}) and (\ref{ub13}) we have
    \begin{align}\label{ub14}
      g(\Lambda_{k,q})\leq& g(\Lambda_{2^{s+1}(k_i+1)-1,q})\nonumber\\
      \leq& 2^{s+1}(k_i+5)\nonumber\\
      \leq& 2k+2^{s+3}\nonumber\\
      \leq& 2k+4T_q+\frac{16T_q-24}{k_i+1}.
    \end{align}
     From (\ref{ub02}) and $k\geq 8T_q^2-10T_q-3$ we see also
     $$T_qk_i+5T_q-5\geq k_{i+1}\geq 2k+1\geq 16T_q^2-20T_q-5$$
     and thus we have $(16T_q-24)/(k_i+1)\leq 1$.
     Therefore, from $T_q>1$, (\ref{ub12}) and (\ref{ub14}) we see that (\ref{ub11}) is valid.
     \end{proof}

\section{Concluding Remarks}
Conjecture~1 was shown to be valid in \cite{Fredi95} for the case $\frac{k+5}{2}|(q-1)$ based on the existence of a special automorphism of $D(k,q)$,
in \cite{XWY14} for the case $\frac{k+5}{2}$ is a power of $p$ based on the existence of backtrackless circuit of type $(1,1,\ldots,1,1)$,
and in \cite{XWY14} for the case $\frac{k+5}{2}|_p(q-1)$ based on the existence of backtrackless circuit of type $(1,1,b,b,\ldots,b^n,b^n)$, respectively,
where $p$ is the characteristic of $\mathbb{F}_q$. A few new results on the girth of $D(k,q)$ are obtained in the present paper based on the existence of backtrackless circuit of type
$(u_1,v_1,\ldots,u_{2n},v_{2n})$ with (\ref{abs_v_eq1}). For example, Conjecture~1 is shown to be valid in Theorem~\ref{theo50} for a new class of infinite pairs $(k,q)$: $p\geq 3$,
$k=2^{s+2}(2t-1)-5$ for positive integers $s,t$ with $2^s\|(q-1)$ and $(2t-1)|_p(q-1)$.
Almost the recent progresses made on the study of the girth of $D(k,q)$ rely heavily on the computation of the homogeneous polynomial $\rho_s(\omega_1,\ldots,\omega_n)$.
It is then of great interest to investigate the properties of these polynomials further in future.


\begin{thebibliography}{10}
	
 \bibitem{Wenger91} R. Wenger, Extremal graphs with no $C^4s$, $C^6s$, or $C^{10}s$, J. Combin. Theory Ser. B 52(1) (1991) 113-116.

 \bibitem{Lazebnik93} F. Lazebnik, V.A. Ustimenko, New examples of graphs without small cycles and of large size, European J. Combin. 14 (1993) 445-460.

 \bibitem{Lazebnik95} F. Lazebnik, V.A. Ustimenko, Explicit construction of graphs with an arbitrary large girth and of large size, Discrete Appl. Math. 60(1995) 275-284.

 \bibitem{LazebnikUstimenko95} F. Lazebnik, V.A. Ustimenko, A.J. Woldar, A new series of dense graphs of high girth, Bull. Amer. Math. Soc. 32(1) (1995) 73-79.

 \bibitem{Fredi95} Z. F\"{u}redi, F. Lazebnik, A. Seress, V.A. Ustimenko, A.J. Woldar, Graphs of pescribed girth and bi-degree, J. Combin. Theory, Ser.\,B 64(2) (1995) 228-239.


 \bibitem{Lazebnik96} F. Lazebnik, V.A. Ustimenko, A.J. Woldar, A characterization of the components of the graphs $D(k,q)$, Discrete Math. 157 (1996) 271-283.

 \bibitem{Lazebnik97} F. Lazebnik, V.A. Ustimenko, A.J. Woldar, New upper bounds on the order of cages, Electron. J. Combin. 14(R13) (1997) 1-11.

 \bibitem{Lazebnik01} F. Lazebnik, A.J. Woldar, General properties of families of graphs defined by some systems of equations, J. Graph Theory 38(2) (2001) 65-86.

 \bibitem{Lazebnik02} F. Lazebnik, R. Viglione, An infinite series of regular edge- but vertex-transitive graphs, J. Graph Theory 41(4) (2002) 249-258.

 \bibitem{Ustimenko03} V.A. Ustimenko, A.J. Woldar, Extremal properties of regular and affine generalized m-gons as tactical configurations, European J. Combin. 24 (2003) 99-111.

 \bibitem{Lazebnik04} F. Lazebnik, R. Viglione, On the connectivity of certain graphs of high girth, Discrete Math. 277 (2004) 309-319.

 \bibitem{Kim04} J. Kim, U. Peled, I. Perepelitsa, V. Pless, S. Friedland, Explicit construction of families of LDPC codes with no 4-cycles, IEEE Trans. Inform. Theory 50(10) (2004) 2378-2388.

 \bibitem{Dmytrenko07} V. Dmytrenko, F. Lazebnik, J. Williford, On monomial graphs of girth eight, Finite Fields. Appl. 13 (2007) 828-842.

 \bibitem{Futorny07} V. Futorny, V. Ustimenko, On small world semiplanes with generalised Schubert cells, Acta Appl. Math. 98 (2007) 47-61.

 \bibitem{Ustimenko13} V.A. Ustimenko, On linguistic dynamical systems, families of graphs of large girth, and cryptography, J. Math. Sci. 140(3) (2007) 461-471.

 \bibitem{Ustimenko09} V.A. Ustimenko, On the homogeneous algebraic graphs of large girth and their applications, Linear Algebra Appl. 430(7) (2009) 1826-1837.

 \bibitem{Yan11} T. Yan, Y. Tang, Constructions of LDPC codes based on polarity graphs with prescribed girth, in: Asia-Pacific Youth Conference on Communication, 2011APYCC, 2011, pp. 60C62.

 \bibitem{Polak13} M. Polak, U. Romanczuk, V. Ustimenko, A. Wroblewska, On the applications of extremal graph theory to coding theory and cryptography, Electron. Notes Discrete Math. 43 (2013) 329-342.

 \bibitem{XWY14} X. Cheng, W. Chen, Y. Tang, On the girth of the bipartite graph $D(k,q)$, Discrete Math. 335 (2014) 25-34.

 \bibitem{XWY16} X. Cheng, W. Chen, Y. Tang, On the conjecture for the girth of the bipartite graph $D(k,q)$, Discrete Math. 335 (2016) 2384-2392.

 \bibitem{Dehghan20} A. Dehghan, A. H. Banihashemi, Counting Short Cycles in Bipartite Graphs: A Fast Technique/Algorithm and a Hardness Result, IEEE Trans. Commun. 68(3) (2020) 1378-1390.

 \bibitem{XCT22} M. Xu, X. Cheng, Y. Tang, On the girth cycles of the bipartite graph $D(k,q)$, submitted to Discrete Math. for publication, https://doi.org/10.48550/arXiv.2207.12752
\end{thebibliography}
\end{document}